\documentclass[letterpaper,11pt,twoside,keywordsasfootnote,addressatend,noinfoline]{article}
\usepackage{fullpage}
\usepackage[english]{babel}
\usepackage{amssymb}
\usepackage{amsmath}
\usepackage{theorem}
\usepackage{epsfig}
\usepackage{subfigure}
\usepackage{mathrsfs}
\usepackage{color}
\usepackage{imsart}

\theorembodyfont{\normalfont}
\newtheorem{theorem}{Theorem}
\newtheorem{lemma}[theorem]{Lemma}

\newenvironment{proof}{\noindent{\bf Proof.}}{\hspace*{2mm}~$\square$}
\newenvironment{proofof}[1]{\noindent{\bf Proof of #1.}}{\hspace*{2mm}~$\square$}

\newcommand{\N}{\mathbb{N}}
\newcommand{\Z}{\mathbb{Z}}
\newcommand{\R}{\mathbb{R}}
\newcommand{\F}{\mathscr F}
\newcommand{\G}{\mathscr G}
\newcommand{\V}{\mathscr V}
\newcommand{\E}{\mathscr E}
\newcommand{\D}{\mathscr D}
\newcommand{\bphi}{\Phi}
\newcommand{\ind}{\mathbf{1}}
\newcommand{\ep}{\epsilon}

\newcommand{\n}{\hspace*{-5pt}}

\DeclareMathOperator{\card}{card}

\DeclareMathOperator{\poisson}{Poisson \,}

\begin{document}

\begin{frontmatter}
\title     {The role of cooperation in spatially explicit \\ economical systems}
\runtitle  {The role of cooperation in spatially explicit economical systems}
\author    {Nicolas Lanchier\thanks{Research supported in part by NSA Grant MPS-14-040958.} and Stephanie Reed}
\runauthor {N. Lanchier and S. Reed}
\address   {School of Mathematical and Statistical Sciences \\ Arizona State University \\ Tempe, AZ 85287, USA.}

\maketitle

\begin{abstract} \ \
 This paper is concerned with a model in econophysics, the subfield of statistical physics that applies concepts from traditional physics to economics. In our model, economical agents are represented by the vertices of a connected graph and are characterized by the number of coins they possess. Agents independently spend one coin at rate one for their basic need, earn one coin at a rate chosen independently from a fixed distribution~$\phi$ and exchange money at rate~$\mu$ with one of their nearest neighbors, with the richest neighbor giving one coin to the other neighbor. If an agent needs to spend one coin when her fortune is at zero, she dies, i.e., the corresponding vertex is removed from the graph. Our first results focus on the two extreme cases of lack of cooperation~$\mu = 0$ and perfect cooperation~$\mu = \infty$ for finite connected graphs. These results suggest that, when overall the agents earn more than they spend, cooperation is beneficial for the survival of the population, whereas when overall the agents earn less than they spend, cooperation becomes detrimental. The infinite one-dimensional system is also studied. In this case, when the agents earn less than they spend in average, the density of agents that die eventually is bounded from below by a positive constant that does not depend on the initial number of coins per agent or the level of cooperation.
\end{abstract}



\end{frontmatter}


\section{Introduction}

\indent Models in econophysics typically consist of large systems of economical agents who earn, spend and exchange money.
 For a review of such models, we refer the reader to~\cite{yakovenko_et_al_2009}.
 These models so far have mainly been studied by statistical physicists.
 From a mathematical point of view, they fall into the category of stochastic processes known as interacting particle systems~\cite{harris_1972, liggett_1985}.
 The most basic model in econophysics has been studied in~\cite{dragulescu_yakovenko_2000} based on numerical simulations but was
 also considered earlier in~\cite{bennati_1988, bennati_1993}.
 This model consists of a system of~$n$ interacting economical agents that are characterized by the number of dollars they possess, and evolves as follows:
 at each time step, an agent chosen uniformly at random gives one dollar to another agent again chosen uniformly at random, unless the first agent has no
 money in which case nothing happens.
 The main conjecture about this model is that, when the number of agents and the money temperature, defined as the average amount of money per agent, are
 both large, the limiting distribution of money is well approximated by the exponential distribution with parameter the money temperature. \\
\indent Spatially explicit versions of this model where agents are located on the vertices of a finite connected graph and can only exchange money with
 their nearest neighbors have been recently introduced and studied analytically in~\cite{lanchier_2017}.
 The non-spatial model considered in~\cite{dragulescu_yakovenko_2000} is simply obtained by assuming that the underlying graph is the complete graph
 with~$n$ vertices.
 It is proved in~\cite{lanchier_2017} that the conjecture in~\cite{dragulescu_yakovenko_2000} is indeed correct and in fact holds for all spatially
 explicit versions, not only the process on the complete graph. \\
\indent In this paper, we study variants of the spatially explicit models~\cite{lanchier_2017} where agents also earn money, spend money
 and die if they run out of money.
 In addition, we assume that the exchange of money occurs in a cooperative setting, meaning that the flow of money is exclusively directed from ``rich''
 agents to ``poor'' agents.
 We also follow the framework of interacting particle systems~\cite{liggett_1985} by assuming that the process evolves in continuous rather
 than in discrete time.
 This approach will allow us to define the system on infinite graphs using an idea of Harris~\cite{harris_1972} that consists in constructing the
 process from a collection of independent Poisson processes. \vspace*{8pt}


\noindent {\bf Model description.}
 To define our spatial model formally, we let~$\G = (\V, \E)$ be a finite or infinite locally finite connected graph.
 Each vertex represents an economical agent who is either alive and characterized by the amount of money she possesses, or dead.
 To fix the ideas, we assume that the amount of money agents who are alive possess is a nonnegative integer representing a number of credits or coins,
 while we use the state~$-1$ for dead agents.
 In particular, the state of the system at time~$t$ is a spatial configuration
 $$ \xi_t : \V \longrightarrow \{-1, 0, 1, 2, \ldots \} $$
 with the value of~$\xi_t (x)$ indicating that agent~$x$ is dead or representing the number of coins this agent possesses when she is alive.
 To define the evolution rules, we attach to each vertex~$x \in \V$ a random variable~$\phi_x$ chosen independently from a fixed distribution~$\phi$.
 The individual at vertex~$x$ earns one coin at rate~$\phi_x$ and, to ensure her survival, spends one coin at rate one.
 The population is also characterized by its level of cooperation which is measured using a nonnegative parameter~$\mu$ as follows:
 nearest neighbors that are alive interact at rate~$\mu$ and, in case one neighbor has at least two more coins than the other neighbor, she gives one
 coin to the other neighbor.
 In particular, the ``richest'' agent before the interaction does not give any coin if this makes her ``poorer'' than her neighbor.
 Finally, if an individual has zero coin at the time she needs to spend one coin then she dies and the corresponding vertex is removed from the graph.
 To describe the dynamics formally, for each spatial configuration~$\xi$, we let
 $$ \begin{array}{rlcl}
    \hbox{\bf spending} & \xi_x^- (z) = \xi (z) - \ind \{z = x \} & \hbox{for all} & z \in \V \vspace*{4pt} \\
     \hbox{\bf earning} & \xi_x^+ (z) = \xi (z) + \ind \{z = x \} & \hbox{for all} & z \in \V \end{array} $$
 be the configurations obtained respectively by removing/adding one coin at vertex~$x$.
 Also, for each edge~$(x, y) \in \E$ of the network of interactions, we let
 $$ \begin{array}{rl}
    \hbox{\bf cooperation} & \xi_{(x, y)} (z) = \xi (z) + \ind \{\xi (x) < \xi (y) - 1 \} (\ind \{z = x \} - \ind \{z = y \}) \vspace*{4pt} \\ & \hspace*{90pt}
                                              + \ \ind \{\xi (y) < \xi (x) - 1 \} (\ind \{z = y \} - \ind \{z = x \}) \end{array} $$
 be the configuration obtained by moving one coin from the richer to the poorer vertex if the two vertices are at least two coins apart.
 The dynamics of the system is then described by the Markov generator~$L$ defined on the set of cylinder functions by
 $$ \begin{array}{rcl}
      Lf (\xi) & \n = \n & \sum_{x \in \V} (f (\xi_x^-) - f (\xi)) \,\ind \{\xi (x) \neq -1 \} \vspace*{6pt} \\ &&
                 \hspace*{20pt} + \ \sum_{x \in \V} \phi_x \,(f (\xi_x^+) - f (\xi)) \,\ind \{\xi (x) \neq -1 \} \vspace*{6pt} \\ &&
                 \hspace*{40pt} + \ \sum_{(x, y) \in \E} \mu \,(f (\xi_{(x, y )}) - f (\xi)) \,\ind \{\xi (x) \neq -1, \xi (y) \neq -1 \}. \end{array} $$
 The first sum describes the rate at which vertices spend one coin, the second sum the rate at which they earn one coin, and the third sum the rate
 at which neighbors exchange one coin.
 As previously mentioned, the process is well defined on locally finite graphs, including infinite graphs, and can be constructed from a collection of independent
 Poisson processes.
 More precisely,
\begin{itemize}
 \item for all~$x \in \V$, let~$N_t^- (x)$ be a Poisson process with intensity one, \vspace*{4pt}
 \item for all~$x \in \V$, let~$N_t^+ (x)$ be a Poisson process with intensity~$\phi_x$, \vspace*{4pt}
 \item for all~$(x, y) \in \E$, let~$N_t (x, y)$ be a Poisson process with intensity~$\mu$.
\end{itemize}
 We further assume that these processes are independent.
 This implies that, with probability one, the arrival times are all distinct.
 A general result due to Harris~\cite{harris_1972} then shows that the process can be constructed using the following rules:
\begin{itemize}
 \item At the arrival times of the Poisson process~$N_t^- (x)$, we take one coin from the individual at vertex~$x$ if this individual is still alive. \vspace*{4pt}
 \item At the arrival times of the Poisson process~$N_t^+ (x)$, we give one coin to the individual at vertex~$x$ if this individual is still alive. \vspace*{4pt}
 \item At the arrival times of~$N_t^+ (x, y)$, we move one coin from~$x$ to~$y$ if~$x$ has at least two more coins than~$y$ or one coin from~$y$ to~$x$ if~$y$ has at least two more coins than~$x$.
\end{itemize}
\vspace*{4pt}


\noindent {\bf Main results.}
 To begin with, we compare the two processes with the same earning rates~$\phi_z$ in the absence of cooperation~$\mu = 0$ and in the presence of perfect
 cooperation~$\mu = \infty$ on finite connected graphs to understand whether cooperation is beneficial or detrimental for survival.
 Our first results look at the probability of global survival that we define as
 $$ p_{\mu} (c, (\phi_z)) = P (\xi_t (z) \neq -1 \ \hbox{for all} \ (z, t) \in \V \times \R_+ \,| \,\xi_0 \equiv c) $$
 where~$c$ refers to the common initial number of coins per agent and where the earning rates~$\phi_z$ are independent realizations of the distribution~$\phi$ for all~$z \in \V$.
 Estimates for the probability of global survival can be expressed in terms of the two key quantities
\begin{equation}
\label{eq:2-keys}
  \begin{array}{l} \D = \max_{x \in \V} \sum_{z \in \V} d (x, z) \quad \hbox{and} \quad \bphi = (1/n) \sum_{z \in \V} \phi_z \end{array}
\end{equation}
 where~$d$ refers to the graph distance and~$n$ to the population size.
 Using that, as long as all the agents are alive, the total number of coins on the graph behaves like a random walk that increases at rate~$n \bphi$
 and decreases at rate~$n$ together with the fact that nearest neighbors are at most one coin apart in the presence of perfect cooperation, we get the
 following theorem.
\begin{theorem} --
\label{th:n}
 In the presence of perfect cooperation~$\mu = \infty$,
 $$ p_{\infty} (c, (\phi_z)) \geq \max \,(0, 1 - \bphi^{- (nc - \D + 1)}). $$
\end{theorem}
 The proof relies, among other things, on an application of the optional stopping theorem for martingales.
 The inequality in the statement turns out to be an equality when~$n = 1$.
 In particular, since the system in the absence of cooperation behaves like~$n$ independent copies of a one-person system, the theorem
 also gives the probability of global survival when~$\mu = 0$.
 Using this and some basic algebra, it can be proved that, when~$\bphi > 1$ and~$c$ is large, the probability of global survival is larger in the presence
 of perfect cooperation than in the absence of cooperation.
\begin{theorem} --
\label{th:n-compare}
 Assume that~$\bphi > 1$. Then, there exists~$c_0 < \infty$ such that,
 $$ \begin{array}{rcl}
      p_0 (c, (\phi_z)) & \n = \n & \prod_{z \in \V} \,\max \,(0, 1 - \phi_z^{- (c + 1)}) \vspace*{4pt} \\
                        & \n \leq \n & \max \,(0, 1 - \bphi^{- (nc - \D + 1)}) \leq p_{\infty} (c, (\phi_z)) \quad \hbox{for all} \quad c \geq c_0. \end{array} $$
\end{theorem}
 More generally, we conjecture that, when~$\bphi > 1$, i.e., when overall the agents earn more than they spend, the probability of global survival is larger
 in the presence of perfect cooperation than in the absence of cooperation regardless of the initial value~$c$. \\
\indent We now focus on the two-person system: we set~$\V = \{x, y \}$ and assume that vertices~$x$ and~$y$ are connected by an edge.
 In this case, Theorem~\ref{th:n} implies that when
 $$ \bphi = (1/2)(\phi_x + \phi_y) > 1 \quad \hbox{and} \quad \phi_x < 1 < \phi_y $$
 global survival is possible in the presence of perfect cooperation whereas individual~$x$ dies almost surely in the absence of cooperation, showing again that
 cooperation is beneficial.
 Cooperation, however, becomes detrimental when
 $$ \bphi = (1/2)(\phi_x + \phi_y) < 1 \quad \hbox{and} \quad \phi_x < 1 < \phi_y. $$
 In this case, regardless of the level of cooperation~$\mu$, global survival is not possible so, to measure the effect of cooperation, we study instead
 $$ E_{\mu} (c, (\phi_z)) = E (\card \{z \in \V : \xi_t (z) \neq -1 \ \hbox{for all} \ t \in \R_+ \,| \,\xi_0 \equiv c), $$
 the expected number of individuals that live forever.
 Due to perfect cooperation and the fact that individual~$x$ dies almost surely, it can be proved that the last time both individuals each have one coin
 is almost surely finite and that, between this time and the first time one of the two individuals dies, the process behaves according to a certain seven-state
 Markov chain.
 Using a first-step analysis to study this Markov chain and part of the proof of Theorem~\ref{th:n}, the expected value of the number of individuals that
 live forever can be computed explicitly.
\begin{theorem} --
\label{th:2-compare}
 Assume that~$\V = \E = \{x, y \}$ and that
 $$ \bphi = (1/2)(\phi_x + \phi_y) < 1 \quad \hbox{and} \quad \phi_x < 1 < \phi_y. $$
 Then, letting~$\Psi = 8 + 2 \phi_x + 2 \phi_y$, for all~$c \geq 1$,
 $$ \begin{array}{rcl}
      E_{\infty} (c, \phi_x, \phi_y) & \n = \n & \displaystyle \bigg(\frac{2}{\Psi} \bigg) \bigg(1 - \frac{1}{\phi_y} \bigg) + \bigg(\frac{\phi_y}{\Psi} + \frac{1}{4} \bigg) \bigg(1 - \bigg(\frac{1}{\phi_y} \bigg)^2 \bigg) \vspace*{8pt} \\
                                     & \n < \n & \displaystyle  1 - \bigg(\frac{1}{\phi_y} \bigg)^{c + 1} = E_0 (c, \phi_x, \phi_y). \end{array} $$
\end{theorem}
 Our approach to prove this result works in theory for all complete graphs, but becomes computationally intractable even with only three vertices.
 More generally, we conjecture that, at least on the complete graph and when~$\bphi < 1$, i.e., when overall the agents earn less than they spend, the expected
 number of individuals that live forever is larger in the absence of cooperation than in the presence of perfect cooperation.
 In a nutshell, we conjecture that cooperation is beneficial for populations that are ``productive'' but detrimental for populations that are not. \\
\indent Finally, we look at the infinite system in one dimension: the underlying graph is represented by the integers with each integer being connected to its
 predecessor and to its successor.
 In this case, the process is more difficult to study because the graph is infinite.
 The next result shows that, when the expected value of~$\phi$ is less than one, the density of individuals who die eventually in the infinite one-dimensional system
 is bounded from below by a positive constant that does
 not depend on the initial number of coins per agent.
\begin{theorem} --
\label{th:1D-sink}
 Assume that~$E (\phi) < 1$. Then,
 $$ \lim_{n \to \infty} \ \frac{1}{2n + 1} \ \sum_{z = -n}^n \,\ind \{\xi_t (z) = -1 \ \hbox{for some} \ t \} = l $$
 where~$l > 0$ does not depend on the initial fortune~$c$ per vertex.
\end{theorem}
 To prove this result, we first identify a collection of events that ensure that a given agent dies before time one.
 This, together with the ergodic theorem, implies that the density of agents that die before time one is positive.
 This density, however, depends \emph{a priori} on the initial fortune.
 Then, we define a sink as a vertex such that the agents in any finite interval that contains this vertex earn overall less than they spend.
 The law of large numbers implies that the density of sinks is bounded from below by a constant that does not depend on the initial fortune.
 Using finally that, at time one, each sink is located between two agents who already died, we use a recursive argument to prove that each sink
 dies eventually.
 In conclusion, the density of individuals who die eventually is bounded from below by the density of sinks which, in turn, is bounded from below
 by a positive constant that does not depend on the initial fortune.
 This gives the result. \\
\indent The proof of Theorem~\ref{th:1D-sink} also suggests that, when the expected value of~$\phi$ is larger than one, the density of agents who live
 forever can be made arbitrarily close to one by choosing the initial fortune~$c$ large enough.
 The proof of this result, however, requires additional arguments that we were not able to make rigorous.


\section{Proof of Theorems~\ref{th:n} and~\ref{th:n-compare}}
\label{sec:n}

\indent In this section, we start by collecting some preliminary results about martingales that will be used later to prove the first two theorems.
 The first step is to estimate probabilities related to the continuous-time Markov chain~$(W_t)$ with transition rates
\begin{equation}
\label{eq:rates}
  \begin{array}{rcl}
  \lim_{\ep \to 0} \,\ep^{-1} \,P (W_{t + \ep} = W_t + 1) & \n = \n & \sum_{z \in \V} \phi_z \vspace*{4pt} \\
  \lim_{\ep \to 0} \,\ep^{-1} \,P (W_{t + \ep} = W_t - 1) & \n = \n & \card (\V) = n. \end{array}
\end{equation}
 Recall from~\eqref{eq:2-keys} that~$\bphi = (1/n) \sum_{z \in \V} \phi_z$.
\begin{lemma} --
\label{lem:martingale}
 The process~$(\bphi^{- W_t})$ is a martingale.
\end{lemma}
\begin{proof}
 Letting~$(\F_t)$ denote the natural filtration of the process~$(W_t)$, and recalling the expression of its transition rates in~\eqref{eq:rates}, we get
 $$ \begin{array}{l}
    \lim_{\ep \to 0} \,\ep^{-1} \,E (\bphi^{- W_{t + \ep}} - \bphi^{- W_t} \,| \,\F_t) \vspace*{4pt} \\ \hspace*{25pt} =
   (\sum_{z \in \V} \phi_z)(\bphi^{- (W_t + 1)} - \bphi^{- W_t}) + n \,(\bphi^{- (W_t - 1)} - \bphi^{- W_t}) \vspace*{4pt} \\ \hspace*{25pt} =
     n \,\bphi \,(\bphi^{- (W_t + 1)} - \bphi^{- W_t}) + n \,(\bphi^{- (W_t - 1)} - \bphi^{- W_t}) \vspace*{4pt} \\ \hspace*{25pt} =
     n \,(\bphi^{- W_t} - \bphi^{- (W_t - 1)}) + n \,(\bphi^{- (W_t - 1)} - \bphi^{- W_t}) = 0, \end{array} $$
 which shows that~$(\bphi^{- W_t})$ is a martingale.
\end{proof} \\ \\
 To state our next results, we define the stopping times
 $$ T_i = \inf \,\{t : W_t = i \} \quad \hbox{for all} \quad i \in \Z. $$
\begin{lemma} --
\label{lem:ost}
 Assume that~$M \leq nc \leq N$ and~$\bphi \neq 1$. Then,
 $$ p (M, N) = P (T_N < T_M \,| \,W_0 = nc) = \frac{1 - \bphi^{- (nc - M)}}{1 - \bphi^{- (N - M)}}. $$
\end{lemma}
\begin{proof}
 Since~$(\bphi^{- W_t})$ is a martingale and the process stopped at time~$T = \min (T_M, T_N)$ is bounded, we may apply the optional stopping theorem to get
\begin{equation}
\label{eq:ost-1}
  E (\bphi^{- W_T}) = E (\bphi^{- W_0}) = \bphi^{- nc}.
\end{equation}
 Note also that
\begin{equation}
\label{eq:ost-2}
 \begin{array}{rcl}
   E (\bphi^{- W_T}) & \n = \n & E (\bphi^{- W_T} \,| \,T = T_M)(1 - p (M, N)) + E (\bphi^{- W_T} \,| \,T = T_N) \,p (M, N) \vspace*{4pt} \\
                     & \n = \n & \bphi^{- M} \,(1 - p (M, N)) + \bphi^{- N} \,p (M, N) \vspace*{4pt} \\
                     & \n = \n & (\bphi^{- N} - \bphi^{- M}) \,p (M, N) + \bphi^{- M}. \end{array}
\end{equation}
 Combining~\eqref{eq:ost-1} and~\eqref{eq:ost-2}, we conclude that
 $$ p (M, N) = \frac{\bphi^{- nc} - \bphi^{- M}}{\bphi^{- N} - \bphi^{- M}} = \frac{1 - \bphi^{- (nc - M)}}{1 - \bphi^{- (N - M)}}. $$
 This completes the proof.
\end{proof}
\begin{lemma} --
\label{lem:monotone-cv}
 For all~$M \leq nc$ and all~$\bphi > 0$,
 $$ q (M) = P (T_M = \infty \,| \,W_0 = nc) = \max \,(0, 1 - \bphi^{- (nc - M)}). $$
\end{lemma}
\begin{proof}
 We distinguish three cases depending on the value of~$\bphi$.
\begin{itemize}
 \item When~$\bphi = 1$, the process~$(W_t)$ is the one-dimensional symmetric random walk which is known to be recurrent. This gives the probability~$P (T_M = \infty) = 0$. \vspace*{4pt}
 \item When~$\bphi < 1$, the law of large numbers implies that~$W_t \to - \infty$ almost surely.
       In particular, the stopping time~$T_M$ is again almost surely finite and the probability~$q (M) = 0$. \vspace*{4pt}
 \item When~$\bphi > 1$, the law of large numbers now gives~$W_t \to \infty$ so
       $$ \{T_M = \infty \} = \{T_N < T_M \ \hbox{for all} \ N \geq nc \} \quad \hbox{almost surely}. $$
       Since in addition we have the inclusions
       $$ \{T_{N + 1} < T_M \} \subset \{T_N < T_M \} \quad \hbox{for all} \quad N \geq nc, $$
       by monotone convergence and Lemma~\ref{lem:ost}, we get
       $$ \begin{array}{rcl}
            q (M) & \n = \n & P (T_N < T_M \ \hbox{for all} \ N \geq nc \,| \,W_0 = nc) \vspace*{4pt} \\
                  & \n = \n & P (\lim_{N \to \infty} \,\{T_N < T_M \} \,| \,W_0 = nc) \vspace*{4pt} \\
                  & \n = \n & \lim_{N \to \infty} \,P (T_N < T_M \,| \,W_0 = nc) = 1 - \bphi^{- (nc - M)}. \end{array} $$
\end{itemize}
 Observing also that~$1 - \bphi^{- (nc - M)} \leq 0$ if and only if~$\bphi \leq 1$ gives the result.
\end{proof} \\ \\
 Lemma~\ref{lem:monotone-cv} is the main ingredient to prove Theorem~\ref{th:n}.
 To see the connection between the previous martingale results and the economical system, define
 $$ \begin{array}{l} \tau = \inf \,\{t : \xi_t (x) = - 1 \ \hbox{for some} \ x \in \V \} \quad \hbox{and} \quad Z_t = \sum_{z \in \V} \xi_t (z) \end{array} $$
 and observe that, before time~$\tau$, the individual at~$z$ is alive, earns one coin at rate~$\phi_z$ and spends one coin at rate one, therefore
 $$ \begin{array}{rcl}
    \lim_{\ep \to 0} \,\ep^{-1} \,P (Z_{t + \ep} = Z_t + 1 \,| \,\tau > t) & \n = \n & \sum_{z \in \V} \phi_z \vspace*{4pt} \\
    \lim_{\ep \to 0} \,\ep^{-1} \,P (Z_{t + \ep} = Z_t - 1 \,| \,\tau > t) & \n = \n & \card (\V) = n. \end{array} $$
 In other words, by time~$\tau$, the total number of coins behaves like the Markov chain~$(W_t)$.
 Using this and the previous lemma, we can now prove the theorem. \\ \\
\begin{proofof}{Theorem~\ref{th:n}}
 In the limiting case~$\mu = \infty$ and as long as all the individuals are alive, each time an individual has at least two more coins than one of her neighbors, this individual instantaneously
 gives a coin to one of her poorest neighbors, therefore 
 $$ |\xi_t (x) - \xi_t (y)| \leq 1 \quad \hbox{for all} \quad (x, y) \in \E \ \hbox{and} \ t < \tau. $$
 Now, letting~$x, y \in \V$ be arbitrary, there exist
 $$ z_0 = x, z_1, \ldots, z_d = y \in \V \quad \hbox{such that} \quad (z_i, z_{i + 1}) \in \E \ \ \hbox{for all} \ \ i = 0, 1, \ldots, d - 1 $$
 where~$d = d (x, y)$.
 In particular, the triangle inequality implies that
\begin{equation}
\label{eq:triangle}
  \begin{array}{rcl}
    |\xi_t (x) - \xi_t (y)| & \n \leq \n & |\xi_t (z_0) - \xi_t (z_1)| + \cdots + |\xi_t (z_{d - 1}) - \xi_t (z_d)| \vspace*{4pt} \\
                            & \n \leq \n &   d = d (x, y) \end{array}
\end{equation}
 for all~$t < \tau$.
 Now, on the event that~$\tau < \infty$, just before that time, there is at least one vertex, say~$x$, with zero coin, while the other vertices have a positive fortune.
 This, together with~\eqref{eq:triangle}, implies that the total number of coins satisfies
 $$ \begin{array}{l} Z_{\tau-} = \sum_{z \in \V} \xi_{\tau-} (z) = \sum_{z \in \V} |\xi_{\tau-} (x) - \xi_{\tau-} (z)| \leq \sum_{z \in \V} d (x, z). \end{array} $$
 Taking the maximum over all possible configurations gives
 $$ \begin{array}{l} Z_{\tau-} \leq \max_{x \in \V} \sum_{z \in \V} d (x, z) = \D. \end{array} $$
 Finally, using Lemma~\ref{lem:monotone-cv} and observing that all the individuals survive if and only if~$\tau = \infty$ give the following lower bound for the probability of global survival
 $$ \begin{array}{rcl}
     p_{\infty} (c, (\phi_z)) & \n = \n & P (\tau = \infty \,| \,\xi_0 (z) = c \ \hbox{for all} \ z \in \V) \vspace*{4pt} \\
                              & \n \geq \n & P (Z_t \geq \D \ \hbox{for all} \ t \,| \,\xi_0 (z) = c \ \hbox{for all} \ z \in \V) \vspace*{4pt} \\
                              & \n = \n & P (W_t > \D - 1 \ \hbox{for all} \ t \,| \,W_0 = nc) \vspace*{4pt} \\
                              & \n = \n & P (T_{\D - 1} = \infty \,| \,W_0 = nc) = q (\D - 1) \vspace*{4pt} \\
                              & \n = \n & \max \,(0, 1 - \bphi^{- (nc - \D + 1)}). \end{array} $$
 This completes the proof of the theorem.
\end{proofof} \\ \\
 Using Lemma~\ref{lem:monotone-cv} and Theorem~\ref{th:n}, we can now prove Theorem~\ref{th:n-compare}. \\ \\
\begin{proofof}{Theorem~\ref{th:n-compare}}
 It follows from Lemma~\ref{lem:monotone-cv} that, in the presence of only one vertex, say~$x$, the probability of survival is given by
 $$ p_0 (c, \phi_x) = q (- 1) = \max \,(0, 1 - \phi_x^{- (c + 1)}). $$
 Since in the absence of cooperation~$\mu = 0$, the system with~$n$ individuals consists of~$n$ independent copies of a one-person system, we get
 $$ \begin{array}{l} p_0 (c, (\phi_z)) = \prod_{z \in \V} \,p_0 (c, \phi_z) = \prod_{z \in \V} \,\max \,(0, 1 - \phi_z^{- (c + 1)}). \end{array} $$
 It directly follows that
 $$ p_0 (c, (\phi_z)) = 0 \quad \hbox{when} \quad \phi_z \leq 1 \quad \hbox{for some} \quad z \in \V $$
 so the inequality to be proved is obvious in this case.
 Assume now that~$\phi_z > 1$ for all~$z \in \V$.
 In this case, we have the following inequalities:
 $$ \begin{array}{rcl}
    \log \,(p_0 (c, (\phi_z))) & \n = \n & \sum_{z \in \V} \log \,(1 - \phi_z^{- (c + 1)}) \leq - \sum_{z \in \V} \phi_z^{- (c + 1)} \vspace*{4pt} \\
    \log \,(p_{\infty} (c, (\phi_z))) & \n \geq \n & \log (1 - \bphi^{- (nc - \D + 1)}) \geq - \bphi^{- (nc - \D + 1)} / (1 - \bphi^{- (nc - \D + 1)}). \end{array} $$
 In particular, since~$\bphi > 1$, for all~$n \geq 2$ and~$c$ sufficiently large,
 $$ \begin{array}{rcl}
    \log \,(p_{\infty} (c, (\phi_z))) & \n \geq \n & - \bphi^{- (nc - \D + 1)} / (1 - \bphi^{- (nc - \D + 1)}) \vspace*{4pt} \\
                                      & \n \geq \n & - 2 \,\bphi^{- (nc - \D + 1)} \geq - 2 \,(\min_{z \in \V} \phi_z)^{- (nc - \D + 1)} \vspace*{4pt} \\
                                      & \n \geq \n & - (\min_{z \in \V} \phi_z)^{- (c + 1)} \geq - \sum_{z \in \V} \phi_z^{- (c + 1)} \vspace*{4pt} \\
                                      & \n \geq \n & \log \,(p_0 (c, (\phi_z))). \end{array} $$
 This completes the proof of the theorem.
\end{proofof}

\section{Proof of Theorem~\ref{th:2-compare}}
\label{sec:two-person}

\indent As stated in the introduction, the two-person system is simple enough that we may calculate certain probabilities by hand.
 Since there are only two vertices, we will call them~$x$ and~$y$ and the rates at which they earn a coin~$\phi_x$ and~$\phi_y$, respectively.
 To simplify the notation, write
 $$ X_t = \xi_t (x) \quad \hbox{and} \quad Y_t = \xi_t (y) \quad \hbox{for all} \quad t \geq 0. $$
 Letting~$T_- = \inf \,\{t : \min (X_t, Y_t) = - 1 \}$, the process
 $$ \bphi^{- (X_{t \wedge T_-} + Y_{t \wedge T_-})} = \bigg(\frac{2}{\phi_x + \phi_y} \bigg)^{X_{t \wedge T_-} + Y_{t \wedge T_-}} $$
 is again a martingale.
 Using that the individuals' fortune differ by at most one coin in the presence of perfect cooperation, and repeating the proofs of Lemmas~\ref{lem:ost}
 and~\ref{lem:monotone-cv}, we easily show that, when both individuals start with~$c$ coins, the probability of global survival satisfies
 $$ \begin{array}{rcl}
      p_{\infty} (c, \phi_x, \phi_y) & \n = \n & P (\min (X_t, Y_t) \geq 0 \ \hbox{for all} \ t \,| \,X_0 = Y_0 = c) \vspace*{4pt} \\
                                     & \n \geq \n & P (X_t + Y_t > 0 \ \hbox{for all} \ t \,| \,X_0 = Y_0 = c) \vspace*{4pt} \\
                                     & \n = \n & \max \,(0, 1 - (2 / (\phi_x + \phi_y))^{2c}) \end{array} $$
 in the case of perfect cooperation.
 In particular, when
 $$ \phi_x + \phi_y > 2 \quad \hbox{and} \quad \phi_x < 1 < \phi_y, $$
 while individual~$x$ dies almost surely in the absence of cooperation, global survival is possible in the presence of perfect cooperation,
 showing that cooperation is beneficial in this case.
 We now focus on the parameter region
\begin{equation}
\label{eq:poor-rich}
  \phi_x + \phi_y < 2 \quad \hbox{and} \quad \phi_x < 1 < \phi_y
\end{equation}
 and show that, in this case, cooperation is detrimental:
 individual~$x$ again dies almost surely while individual~$y$ is more likely to live forever in the absence of cooperation than in the presence of
 perfect cooperation.
 The probability of survival can be computed explicitly. \\
\indent Using again that the individuals' fortune differ by at most one coin in the presence of perfect cooperation, together with the fact that
 global survival is not possible when~\eqref{eq:poor-rich} holds, implies that the stopping time~$T_-$ is almost surely finite and that
 $$ (X_{T_-}, Y_{T_-}) \in \{(-1, 0), (-1, 1), (0, -1), (1, -1) \}. $$
 To simplify the notation, we rename these four states as well as the three adjacent states as shown in Figure~\ref{fig:transitions} and define
 the stopping times and corresponding probabilities
 $$ \tau_i = \inf \,\{t : (X_t, Y_t) = S_i \} \quad \hbox{and} \quad p_i = P (T_- = \tau_i) \quad \hbox{for} \ i = 1, 2, 3, 4. $$
 The probabilities~$p_i$ are computed explicitly in the next lemma.
\begin{lemma} --
\label{lem:two-person}
 Assume~\eqref{eq:poor-rich} and perfect cooperation. Then,
 $$ p_1 = p_2 = 2 / \Psi \qquad p_3 = \phi_x / \Psi + 1/4 \qquad p_4 = \phi_y / \Psi + 1/4 $$
 where~$\Psi = 8 + 2 \phi_x + 2 \phi_y$.
\end{lemma}
\begin{figure}[t]
 \centering
 \scalebox{0.50}{\input{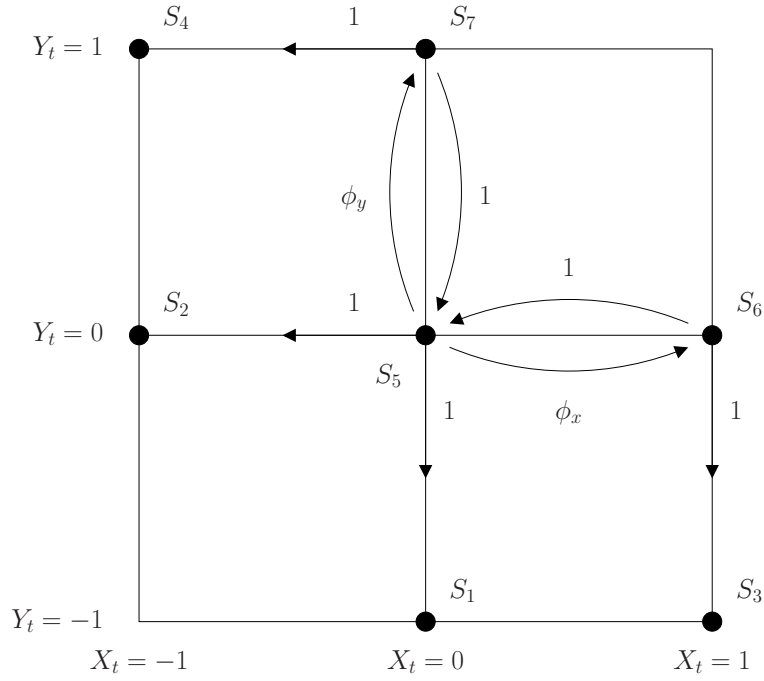}}
 \caption{\upshape The seven states and transition rates between times~$T_+$ and~$T_-$.}
\label{fig:transitions}
\end{figure}
\begin{proof}
 Observe that~$T_-$ is almost surely finite when~\eqref{eq:poor-rich} holds.
 Since in addition the individuals' fortune differ by at most one coin before time~$T_-$,
 $$ T_+ = \sup \,\{t : X_t = Y_t = 1 \} < \infty \quad \hbox{almost surely}. $$
 Also, between time~$T_+$ and time~$T_-$, the process consists of the seven-state continuous-time Markov chain whose transition rates are indicated in Figure~\ref{fig:transitions}.
 Referring again to the picture for the name of the states, we define the conditional probabilities
 $$ p_{ij} = P (T_- = \tau_i \,| \,(X_0, Y_0) = S_j) \quad \hbox{for all} \quad (i, j) \in \{1, 2, 3, 4 \} \times \{5, 6, 7 \}. $$
 Using a first-step analysis and looking at the probabilities at which the process starting from state~$S_5$ jumps to each of the four adjacent states, we get
 $$ p_{15} = \frac{1}{2 + \phi_x + \phi_y} + \frac{\phi_x \,p_{16}}{2 + \phi_x + \phi_y} + \frac{\phi_y \,p_{17}}{2 + \phi_x + \phi_y}. $$
 The same idea gives~$p_{16} = p_{17} = (1/2) \,p_{15}$.
 Solving the system, we get
 $$ p_{15} = \frac{2}{4 + \phi_x + \phi_y} \quad \hbox{and} \quad p_{16} = p_{17} = \frac{1}{4 + \phi_x + \phi_y}. $$
 Since in addition the first state visited after time~$T_+$ is equally likely to be~$S_6$ and~$S_7$, we conclude that the probability~$p_1$ is given by
 $$ p_1 = \frac{p_{16} + p_{17}}{2} = \frac{1}{4 + \phi_x + \phi_y} = \frac{2}{\Psi} $$
 which, by symmetry, is also the value of~$p_2$.
 To compute~$p_3$, we again use a first-step analysis to obtain a system involving the three conditional probabilities:
 $$ p_{35} = \frac{\phi_x \,p_{36}}{2 + \phi_x + \phi_y} + \frac{\phi_y \,p_{37}}{2 + \phi_x + \phi_y} \qquad p_{36} = \frac{1}{2} + \frac{p_{35}}{2} \qquad p_{37} = \frac{p_{35}}{2}. $$
 Solving the system gives
 $$ p_{35} = \frac{\phi_x}{4 + \phi_x + \phi_y} \qquad p_{36} = \frac{1}{2} + \frac{\phi_x}{8 + 2 \phi_x + 2 \phi_y} \qquad p_{37} = \frac{\phi_x}{8 + 2 \phi_x + 2 \phi_y} $$
 from which it follows as before that
 $$ p_3 = \frac{p_{36} + p_{37}}{2} = \frac{\phi_x}{8 + 2 \phi_x + 2 \phi_y} + \frac{1}{4} = \frac{\phi_x}{\Psi} + \frac{1}{4}. $$
 By symmetry, the value of~$p_4$ is obtained by exchanging the role of~$\phi_x$ and~$\phi_y$ in the previous expression, which completes the proof.
\end{proof} \\ \\
 Using the previous lemma as well as Lemma~\ref{lem:monotone-cv} and conditioning on the first boundary state visited, we deduce that the expected number of individuals that
 survive in the presence of perfect cooperation, which is also the probability that~$y$ survives, is given by
 $$ E_{\infty} (c, \phi_x, \phi_y) = p_2 \,p_0 (0, \phi_y) + p_4 \,p_0 (1, \phi_y)
                                   = \bigg(\frac{2}{\Psi} \bigg) \bigg(1 - \frac{1}{\phi_y} \bigg) + \bigg(\frac{\phi_y}{\Psi} + \frac{1}{4} \bigg) \bigg(1 - \bigg(\frac{1}{\phi_y} \bigg)^2 \bigg). $$
 Since in addition
 $$ 1 - \frac{1}{\phi_y} < 1 - \bigg(\frac{1}{\phi_y} \bigg)^2 \leq 1 - \bigg(\frac{1}{\phi_y} \bigg)^{c + 1} $$
 for all~$\phi_y > 1$ and~$c \geq 1$, and since
 $$ \bigg(\frac{2}{\Psi} \bigg) + \bigg(\frac{\phi_y}{\Psi} + \frac{1}{4} \bigg) = P (T_- = \tau_2 \ \hbox{or} \ T_- = \tau_4) \leq 1, $$
 we conclude that
 $$ E_{\infty} (c, \phi_x, \phi_y) < 1 - \bigg(\frac{1}{\phi_y} \bigg)^{c + 1} = E_0 (c, \phi_x, \phi_y). $$
 This completes the proof of Theorem~\ref{th:2-compare}.

\section{Proof of Theorem~\ref{th:1D-sink}}

\indent As explained in the introduction, the first step to prove Theorem~\ref{th:1D-sink} is to identify a collection of events that simultaneously occur with positive probability and ensure that a given vertex,
 say the origin, dies before time one.
 These events are defined from the collection of independent Poisson processes introduced at the end of the model description as follows:
 $$ \begin{array}{rcl}
      A_1 & \n = \n & \{N_1^+ (0) = 0 \ \hbox{and} \ N_1^- (0) \geq (c + 1)^2 \} \vspace*{4pt} \\
      A_2 & \n = \n & \{N_1^+ (z) = N_1^- (z) = 0 \ \hbox{for all} \ z \in \Z \ \hbox{such that} \ 0 < |z| \leq c + 1 \} \vspace*{4pt} \\
      A_3 & \n = \n & \{N_1^+ (c + 2) + \cdots + N_1^+ (c + n + 1) \leq n \ \hbox{for all} \ n > 0 \} \vspace*{4pt} \\
      A_4 & \n = \n & \{N_1^+ (- (c + 2)) + \cdots + N_1^+ (- (c + n + 1)) \leq n \ \hbox{for all} \ n > 0 \}. \end{array} $$
 The times at which neighbors exchange a coin are unimportant in the proof of the theorem.
 Let~$A$ be the event that consists of the intersection of these four events.
\begin{lemma} --
\label{lem:remove}
 For all~$\mu \in [0, \infty]$, we have~$P (\xi_1 (0) = -1 \,| \,A) = 1$.
\end{lemma}
\begin{proof}
\begin{figure}[t]
 \centering
 \scalebox{0.50}{\input{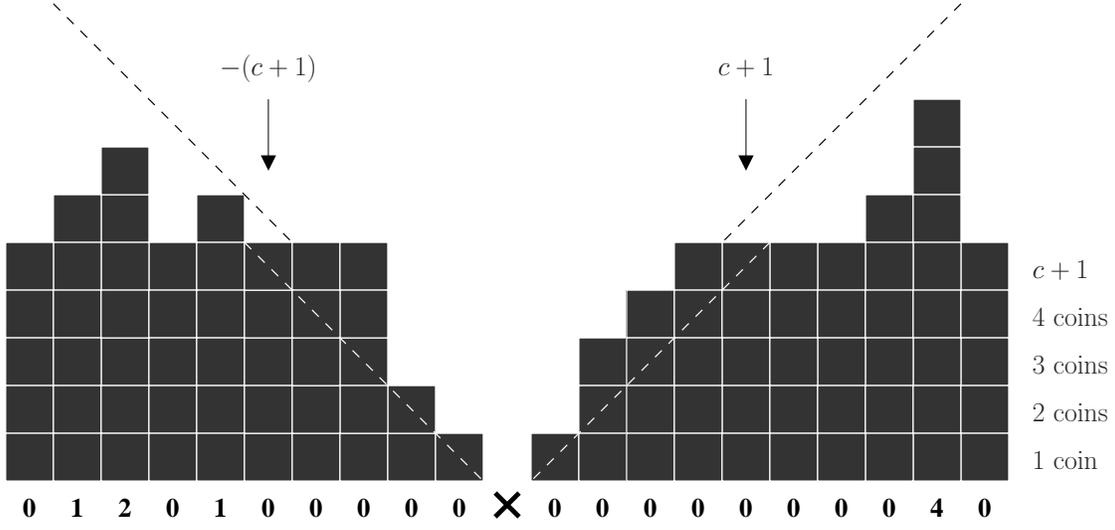}}
 \caption{\upshape Typical configuration at time one when~$A$ occurs: the agent at 0 is dead and the fortune of the agents at distance at least~$c + 2$ from
                   the origin is below the black dashed line.
                   The numbers at the bottom of the picture give the number of coins these agents earned by time one.
                   In the picture, we assume that these agents do not spend any coin, in which case the fortune of the agents within distance~$c + 1$ of the
                   origin is above the white dashed line.}
\label{fig:profile}
\end{figure}
 To begin with, we ignore the exchange of money between~$c + 1$ and its right neighbor and between~$- (c + 1)$ and its left neighbor.
 Recalling that an agent can receive one coin from a neighbor only if this neighbor has at least two more coins, on the event~$A_1 \cap A_2$,
\begin{equation}
\label{eq:remove-1}
  \xi_1 (0) = - 1 \quad \hbox{and} \quad c \geq \xi_t (z) \geq |z| - 1 \quad \hbox{for all} \quad 0 < |z| \leq c + 1 \ \hbox{and} \ t \in (0, 1).
\end{equation}
 Note that the second inequality above becomes an equality when~$\mu = \infty$.
 In this case, the total loss of coins among the~$2c + 3$ vertices around zero is given by
 $$ (c + 1) + 2c + 2 (c - 1) + \cdots + 2 \times 1 + 2 \times 0 = (c + 1)^2, $$
 which explains our definition of the event~$A_1$.
 Observe that~\eqref{eq:remove-1} implies that there are exactly~$c$ coins at vertex~$c + 1$ until time one.
 In particular, looking at the full system and allowing the exchange of money between~$c + 1$ and its right neighbor, on the event~$A_3$,
\begin{equation}
\label{eq:remove-2}
  \begin{array}{l}
    \hbox{number of coins traveling~$c + 1 \to c + 2$ by time one} \vspace*{4pt} \\
    \hspace*{50pt} \geq \ \hbox{number of coins traveling~$c + 2 \to c + 1$ by time one.} \end{array}
\end{equation}
 By symmetry, on the event~$A_4$,
\begin{equation}
\label{eq:remove-3}
  \begin{array}{l}
    \hbox{number of coins traveling~$- (c + 1) \to - (c + 2)$ by time one} \vspace*{4pt} \\
    \hspace*{20pt} \geq \ \hbox{number of coins traveling~$- (c + 2) \to - (c + 1)$ by time one.} \end{array}
\end{equation}
 Combining~\eqref{eq:remove-1}--\eqref{eq:remove-3}, we deduce that given the event~$A$ we must have~$\xi_1 (0) = -1$.
\end{proof} \\ \\
 To prove that the event~$A$ has a positive probability, we let
 $$ \ep = - (1/2)(E (\phi) - 1) > 0 \quad \hbox{so that} \quad E (\phi) = 1 - 2 \ep $$
 and call vertex~$z \in \Z$
 $$ \begin{array}{rcll}
    \hbox{a right $\ep$-sink} & \hbox{when} & \phi_z + \phi_{z + 1} + \cdots + \phi_{z + n} \leq (n + 1)(1 - \ep) & \hbox{for all} \ n \in \N \vspace*{4pt} \\
    \hbox{a left $\ep$-sink} & \hbox{when} & \phi_z + \phi_{z - 1} + \cdots + \phi_{z - n} \leq (n + 1)(1 - \ep) & \hbox{for all} \ n \in \N. \end{array} $$
 Then, we have the following result.
\begin{lemma} --
\label{lem:slln}
 We have~$P (z \ \hbox{is a right $\ep$-sink}) = P (z \ \hbox{is a right $\ep$-sink}) = a > 0$.
\end{lemma}
\begin{proof}
 Define the process
 $$ X_n = X_n (z) = \phi_z + \phi_{z + 1} + \cdots + \phi_{z + n} - (n + 1)(1 - \ep) \quad \hbox{for all} \quad n \in \N. $$
 Since the random variables~$\phi_z, \phi_{z + 1}, \ldots, \phi_{z + n}$ are independent and identically distributed, it follows from the strong law of large numbers that
 $$ \lim_{n \to \infty} \frac{X_n}{n + 1} = \lim_{n \to \infty} \,\frac{1}{n + 1} \,\sum_{i = 0}^n \ (\phi_{z + i} - (1 - \ep)) = E (\phi) - (1 - \ep) = - \ep < 0. $$
 In particular, there exists~$N$, fixed from now on, such that
\begin{equation}
\label{eq:slln-1}
  P (X_n \leq 0 \ \hbox{for all} \ n \geq N) = P \bigg(\sum_{i = 1}^n \ (\phi_{z + i} - (1 - \ep)) \leq 0 \ \hbox{for all} \ n \geq N \bigg) \geq 1/2.
\end{equation}
 In addition, since~$E (\phi) < 1 - \ep$, we have~$p = P (\phi \leq 1 - \ep) > 0$ so
\begin{equation}
\label{eq:slln-2}
  P (X_n \leq 0 \ \hbox{for all} \ n < N) \geq P (\phi_{z + i} \leq 1 - \ep \ \hbox{for all} \ i < N) = p^N > 0.
\end{equation}
 Finally, combining~\eqref{eq:slln-1} and~\eqref{eq:slln-2} and using that the events~$\{X_n \leq 0 \}$ for different values of~$n \in \N$ are positively correlated, we conclude that
 $$ \begin{array}{l}
      P (z \ \hbox{is a right $\ep$-sink}) = P (X_n \leq 0 \ \hbox{for all} \ n \geq 0) \vspace*{4pt} \\ \hspace*{25pt} =
      P (X_n \leq 0 \ \hbox{for all} \ n \geq N \,| \,X_n \leq 0 \ \hbox{for all} \ n < N) \,P (X_n \leq 0 \ \hbox{for all} \ n < N) \vspace*{4pt} \\ \hspace*{25pt} \geq
      P (X_n \leq 0 \ \hbox{for all} \ n \geq N) \,P (X_n \leq 0 \ \hbox{for all} \ n < N) \geq (1/2) \,p^N > 0. \end{array} $$
 It also follows from obvious symmetry that the probability that~$z$ is a left~$\ep$-sink is equal to the probability that it is a right~$\ep$-sink.
 This completes the proof.
\end{proof} \\ \\
 Using the previous lemma, we can now prove that the event~$A$ has positive probability.
\begin{lemma} --
\label{lem:collapse}
 We have~$P (A) > 0$.
\end{lemma}
\begin{proof}
 Since the Poisson processes in the graphical representation are independent
 $$ P (A) = P (A_1) \,P (A_2) \,P (A_3) \,P (A_4). $$
 In addition, for any given~$c$ finite, the first two events have positive probability while, by symmetry, the last two events have the same probability, i.e.,
\begin{equation}
\label{eq:collapse-0}
  P (A_1) \,P (A_2) > 0 \quad \hbox{and} \quad P (A_3) = P (A_4).
\end{equation}
 In particular, to conclude, it suffices to prove that the event~$A_3$ has a positive probability.
 By conditioning on the event that vertex~$c + 2$ is a right $\ep$-sink, we get
\begin{equation}
\label{eq:collapse-1}
  \begin{array}{rcl}
     P (A_3) & \n \geq \n & P (A_3 \,| \,c + 2 \ \hbox{is a right $\ep$-sink}) \,P (c + 2 \ \hbox{is a right~$\ep$-sink}) \vspace*{4pt} \\
             & \n   =  \n & a \,P (A_3 \,| \,c + 2 \ \hbox{is a right $\ep$-sink}) \end{array} 
\end{equation}
 where~$a > 0$ according to Lemma~\ref{lem:slln}. Now, let
 $$ Y_n = \poisson (n \,(1 - \ep)) \ \hbox{be independent for all} \ n > 0. $$
 Using that the events that define the event~$A_3$ are positively correlated and recalling the definition of right~$\ep$-sink, we deduce that
\begin{equation}
\label{eq:collapse-2}
 \begin{array}{l} P (A_3 \,| \,c + 2 \ \hbox{is a right $\ep$-sink}) \geq P (Y_n \leq n \ \hbox{for all} \ n > 0) = \prod_{n > 0} \,P (Y_n \leq n). \end{array}
\end{equation}
 In other respects,
\begin{equation}
\label{eq:collapse-3}
  \begin{array}{rcl}
    \prod_{n > 0} \,P (Y_n \leq n) > 0 & \hbox{if and only if} & \sum_{n > 0} \,- \log (1 - P (Y_n > n)) < \infty \vspace*{4pt} \\
                                       & \hbox{if and only if} & \sum_{n > 0} \,P (Y_n > n) < \infty \end{array}
\end{equation}
 which follows from standard large deviations estimates for the Poisson distribution.
 Combining~\eqref{eq:collapse-1}--\eqref{eq:collapse-3}, we deduce that~$P (A_3) > 0$ which, together with~\eqref{eq:collapse-0}, gives the lemma.
\end{proof} \\ \\
 Since the random variables~$\phi_z$ are independent and identically distributed, we may apply the ergodic theorem together
 with Lemmas~\ref{lem:remove} and~\ref{lem:collapse} to deduce that
\begin{equation}
\label{eq:islands}
  \lim_{n \to \infty} \ \frac{1}{2n + 1} \ \sum_{z = -n}^n \,\ind \{\xi_1 (z) = -1 \} \geq P (A) > 0.
\end{equation}
 Note however that this does not imply our theorem since the probability of~$A_1 \cap A_2$, and therefore the lower bound~$P (A)$, depends on~$c$, the initial number
 of coins per vertex. \\
\indent The second step of the proof is to identify an infinite collection of vertices, that we call $\ep$-sinks, that are removed eventually.
 The density of such vertices is bounded from below by a positive constant that does not depend on~$c$.
 More precisely, we call vertex~$z \in \Z$ an $\ep$-sink if
\begin{equation}
\label{eq:ep-sink}
  \phi_{z - m} + \phi_{z - m + 1} + \cdots + \phi_{z + n} \leq (m + n + 1)(1 - \ep) \quad \hbox{for all} \quad m, n \in \N.
\end{equation}
\begin{lemma} --
\label{lem:ep-sink}
 We have~$P (z \ \hbox{is an~$\ep$-sink}) \geq a^2 > 0$.
\end{lemma}
\begin{proof}
 Let~$A_{m, n}$ be the event in~\eqref{eq:ep-sink} and observe that
 $$ A_{m, 0} \cap A_{0, n} \subset A_{m, n} \quad \hbox{for all} \quad m, n \in \N. $$
 In particular, the event that~$z$ is an~$\ep$-sink is
\begin{equation}
\label{eq:ep-sink-1}
  \bigcap_{m, n} A_{m, n} = \bigcap_{m, n} (A_{m, 0} \cap A_{0, n}) = \bigg(\bigcap_m A_{m, 0} \bigg) \cap \bigg(\bigcap_n A_{0, n} \bigg).
\end{equation}
 Using that~$A_{0, n} = \{X_n \leq 0 \}$ where the process~$(X_n)$ has been defined in the proof of Lemma~\ref{lem:slln} and obvious symmetry, we also have
\begin{equation}
\label{eq:ep-sink-2}
  P \bigg(\bigcap_m A_{m, 0} \bigg) = P \bigg(\bigcap_n A_{0, n} \bigg) = P (X_n \leq 0 \ \hbox{for all} \ n \geq 0) = a > 0
\end{equation}
 according to Lemma~\ref{lem:slln}.
 Combining~\eqref{eq:ep-sink-1} and~\eqref{eq:ep-sink-2}, and using that the events~$A_{m, 0}$ and~$A_{0, n}$ are positively correlated, we conclude that
 $$ P (z \ \hbox{is an~$\ep$-sink}) = P \bigg(\bigcap_{m, n} A_{m, n} \bigg) \geq P \bigg(\bigcap_m A_{m, 0} \bigg) \,P \bigg(\bigcap_n A_{0, n} \bigg) = a^2 > 0. $$
 This completes the proof.
\end{proof} \\ \\
\indent To complete the proof of the theorem, the last step is to show that all the~$\ep$-sinks die eventually with probability one, which is done in the following lemma.
\begin{lemma} --
\label{lem:death-sink}
 Assume that~$x \in \Z$ is an~$\ep$-sink.
 Then~$\xi_t (x) = - 1$ for some~$t$.
\end{lemma}
\begin{proof}
\begin{figure}[t]
 \centering
 \scalebox{0.50}{\input{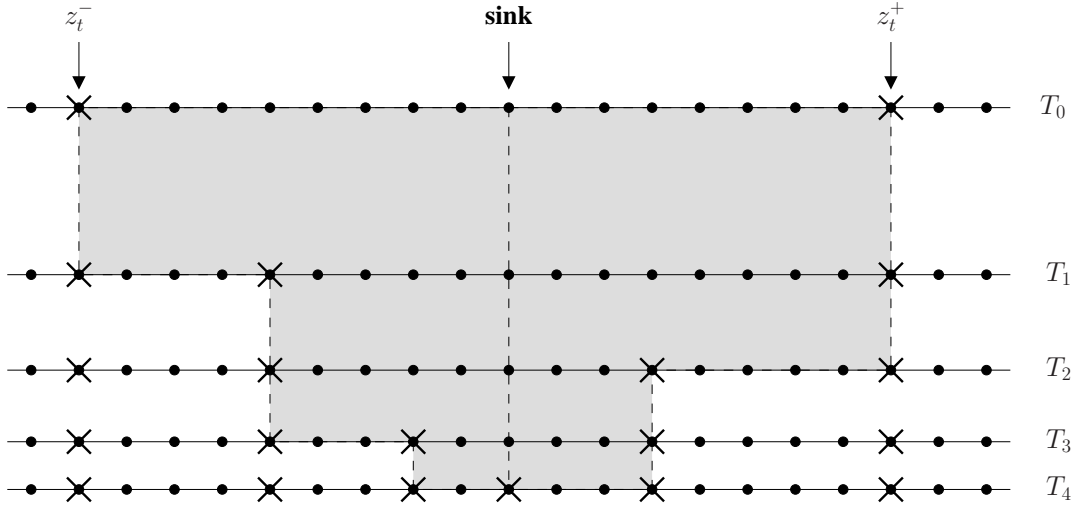}}
 \caption{\upshape Picture of the construction in Lemma~\ref{lem:death-sink} with the sequence of stopping times~$T_i$.
                   The crosses~$\times$ represent the agents that are dead.
                   The gray region shows the interval~$I_t$ from time~$T_0 = 1$ until the sink dies.
                   In our example, it takes four steps to kill the sink located at the center of the picture.}
\label{fig:sink}
\end{figure}
 For all times~$t$, we define
 $$ z_t^- = \sup \,\{z \leq x : \xi_t (z) = - 1 \} \quad \hbox{and} \quad z_t^+ = \inf \,\{z \geq x : \xi_t (z) = - 1 \}. $$
 In view of~\eqref{eq:islands} and since~$-1$ is an absorbing state for each vertex,
 $$ I_t = (z_t^-, z_t^+) \ \ \hbox{is bounded at time~$t = 1$ and nonincreasing in~$t$} $$
 for the inclusion.
 Now, set~$T_0 = 1$ and define recursively
 $$ \begin{array}{rclcl}
      T_i & \n = \n & \inf \,\{t > T_{i - 1} : I_t \neq I_{t-} \} & \hbox{when} & T_{i - 1} < \infty \vspace*{4pt} \\
          & \n = \n & \infty                                      & \hbox{when} & T_{i - 1} = \infty \end{array} $$
 for all~$i > 0$.
 See Figure~\ref{fig:sink} for a picture.
 Given that time~$T_i$ is finite and that the interval~$I_{T_i}$ is nonempty, by the definition of~$\ep$-sink, between time~$T_i$ and time~$T_{i + 1}$, the process
 $$ Z_t = \xi_t (z_{T_i}^- + 1) + \xi_t (z_{T_i}^- + 2) + \cdots + \xi_t (z_{T_i}^+ - 1) $$
 is dominated stochastically by a one-dimensional random walk with a negative drift.
 This implies that the expected number of coins in the interval~$I_t$ is decreasing, therefore one of the vertices in the interval must reach state~$-1$
 in a finite time and
 $$ P (T_{i + 1} < \infty \,| \,T_i < \infty \ \hbox{and} \ I_{T_i} \neq \varnothing) = 1. $$
 Recall also that the interval is bounded at time one and observe that, by definition of the stopping times, the length of the interval decreases
 by at least one at each step, i.e.,
 $$ |I_{T_0}| < \infty \quad \hbox{and} \quad |I_{T_{i + 1}}| \leq |I_{T_i}| - 1 \quad \hbox{when} \quad T_i < T_{i + 1} < \infty. $$
 In summary, it takes only a finite number steps for~$I_t$ to become empty and the duration of each step is almost surely finite.
 Since in addition the sink dies at the time~$I_t$ becomes empty,
 $$ \inf \,\{t : \xi_t (x) = -1 \} = \inf \,\{t : I_t = \varnothing \} < \infty $$
 with probability one.
 This completes the proof.
\end{proof} \\ \\
 As previously, since the random variables~$\phi_z$ are independent and identically distributed, we may apply the ergodic theorem which, together with
 Lemmas~\ref{lem:ep-sink} and~\ref{lem:death-sink}, implies that
 $$ \begin{array}{l}
    \displaystyle \lim_{n \to \infty} \ \frac{1}{2n + 1} \ \sum_{z = -n}^n \,\ind \{\xi_t (z) = -1 \ \hbox{for some} \ t \} \vspace*{0pt} \\ \hspace*{50pt} \geq
    \displaystyle \lim_{n \to \infty} \ \frac{1}{2n + 1} \ \sum_{z = -n}^n \,\ind \{z \ \hbox{is an~$\ep$-sink} \} \geq a^2 > 0. \end{array} $$
 Since~$a$ does not depend on~$c$, this proves Theorem~\ref{th:1D-sink}.

\end{document}